\documentclass[reqno] {amsart}
\usepackage{amsfonts}

\usepackage{amsmath,amssymb}
\newtheorem{theorem}{Theorem}
\theoremstyle{plain}

\newtheorem{lemma}{Lemma}

\numberwithin{equation}{section}

\begin{document}
\title[Hexagonal Fourier Series]{on ces\`{a}ro and Abel-Poisson means of hexagonal fourier series}
\author{Ali Guven}
\address{Department of Mathematics, Faculty of Arts and Sciences, Balikesir
University, 10145 Balikesir, Türk\. {ı}ye.} 
\email{guvennali@gmail.com}
\date{}
\thanks{Research supported by Balikesir University under grant number 2024/019}

\begin{abstract}
Approximation properties of Ces\`{a}ro and Abel-Poisson means of hexagonal Fourier series are studied. The degree of approximation by these means of hexagonal Fourier series of functions, which are continuous and periodic with respect to the hexagon lattice, is estimated in terms of modulus of continuity of functions.
\end{abstract}

\maketitle

\section{Introduction}
Approximation properties of Ces\`{a}ro $(C,\delta)$ and Abel-Poisson means of Fourier series of $2\pi$-periodic functions of a real variable were studied by many mathematicians. The monographs \cite{zhizhiashvili} and \cite{zygmund} contain a lot of results about convergence and the degree of approximation of these means. In same monographs, there are also results regarding approximation properties of Ces\`{a}ro and Abel-Poisson means of functions of two or more variables. Approximation properties of functions of two or more variables are studied usually by assuming that the functions are periodic with respect to each of their variables. In approximation theory of functions of several real variables another definitions of periodicity are also used. The periodicity defined by lattices is the most useful periodicity. This definition of periodicity allows us to study trigonometric approximation problems on non tensor product domains of the Euclidean space, for instance, on regular hexagons in the plane.

 A lattice in the $d-$  dimensional Euclidean space $ \mathbb{R}^{d}$ is the discrete subgroup 
\begin{equation}
    L_{A}:=A\mathbb{Z}^{d}=\lbrace Ak:k\in \mathbb{Z}^{d}\rbrace, \notag
\end{equation}
where $A$ is a $d\times d$ matrix with linearly independent columns, which is called the generator matrix of the lattice $L_{A}$. A bounded set $\Omega \subset \mathbb{R}^{d}$ is said to tile $\mathbb{R}^{d}$ with the lattice $L_{A}$ if
\begin{equation}
    \sum_{k \in \mathbb{Z}^{d}}\chi_{\Omega}(x+Ak)=1\notag
\end{equation}
for almost all $x\in \mathbb{R}^{d}$. The set $\Omega$ is called a spectral set for the lattice $L_{A}$ if it tiles $\mathbb{R}^{d}$ with $L_{A}$. For a given lattice the spectral set is not unique. We fix $\Omega$ such that $\Omega$ contains the point $0$ in its interior and the tiling holds pointwise and without overlapping, that is 
\begin{equation}
    \sum_{k \in \mathbb{Z}^{d}}\chi_{\Omega}(x+Ak)=1\notag
\end{equation}
for all $x\in \mathbb{R}^{d}$ and
\begin{equation}
    (\Omega +Ak)\cap (\Omega+Aj)=\emptyset\notag
\end{equation}
for $k\neq j$.
For example we take $\Omega=\left [-\frac{1}{2},\frac{1}{2}\right )^{d}$ for the standard lattice $L_{I_{d}}=\mathbb{Z}^{d}.$

Let $\Omega$ be a spectral set for the lattice $L_{A}$. Consider the space $L^{2}(\Omega),$ which is a Hilbert space with respect to the inner product
\begin{equation}
     \langle f,g\rangle_{L^{2}(\Omega)}=\frac{1}{\left |\Omega\right |}\int_{\Omega}f(x)\overline{g(x)}dx,\notag
\end{equation}
where $\left |\Omega\right |$ denotes the Lebesgue measure of $\Omega$. The following theorem of B. Fuglede allows us to define Fourier series on the spectral set $\Omega$:
\begin{theorem}[\cite{fuglede}]
    The open bounded set $\Omega$ is a spectral set for the lattice $L_{A}$ if and only if the set
    \begin{equation}
        \lbrace e^{2\pi i\langle A^{-tr}k,x\rangle}:k\in \mathbb{Z}^{d}\rbrace\notag
    \end{equation}
is an orthonormal basis of $L^{2}(\Omega)$.
\end{theorem}
Hence, if  $\Omega$ is a spectral set for the lattice $L_{A}$, the Fourier series of a function $f\in L^{2}(\Omega)$ becomes 
\begin{equation}
    \sum_{k \in \mathbb{Z}^{d}}\widehat{f}_{k}e^{2\pi i\langle A^{-tr}k,x\rangle},\notag
\end{equation}
where 
\begin{equation}
    \widehat{f}_{k}=\frac{1}{\left |\Omega\right |}\int_{\Omega}f(x)e^{-2\pi i\langle A^{-tr}k,x\rangle}dx, \hspace{2mm}k \in \mathbb{Z}^{d}\notag
\end{equation}
are Fourier coefficients.

A function defined on $\mathbb{R}^{d}$ is called periodic with respect to the lattice $L_{A}$ or $A-$periodic if 
\begin{equation}
    f(x+Ak)=f(x)\notag
\end{equation}
for all $k\in \mathbb{Z}^{d}.$ 

More detailed information on Fourier analysis with lattices can be found in \cite{li sun xu}.

\section{Hexagonal Fourier Series}

Besides the rectangular domain, the simplest and the most useful spectral set is a regular hexagon on the plane $\mathbb{R}^{2}$. The generator matrix and the spectral set of the hexagonal lattice $H\mathbb{Z}^{2}$ are given by
\begin{equation}
    H=\begin{pmatrix}
\sqrt{3}&0\\
-1&2
\end{pmatrix} \notag
\end{equation}
and
\begin{equation}
    \Omega_{H}=\lbrace (x_{1},x_{2})\in \mathbb{R}^{2}:-1\leq x_{2},\frac{\sqrt{3}}{2}x_{1}\pm \frac{1}{2}x_{2}< 1\rbrace.\notag
\end{equation}
We denote the plane $t_{1}+ t_{2}+ t_{3}=0$ by $\mathbb{R}_{H}^{3}$, that is 
\begin{equation}
    \mathbb{R}_{H}^{3}=\lbrace (t_{1}, t_{2}, t_{3})\in \mathbb{R}^{3}:t_{1}+ t_{2}+ t_{3}=0\rbrace.\notag
\end{equation}
Elements of the set $\mathbb{R}_{H}^{3}$ are called homogeneous coordinates.
Under the transform
\begin{equation}
    t_{1}:=-\frac{x_{2}}{2}+\frac{\sqrt{3}x_{1}}{2}, t_{2}:=x_{2}, t_{3}:= -\frac{x_{2}}{2}-\frac{\sqrt{3}x_{1}}{2},\notag
\end{equation}
the hexagon $\Omega_{H}$ becomes
\begin{equation}
    \Omega =\lbrace (t_{1}, t_{2}, t_{3})\in \mathbb{R}_{H}^{3}:-1\leq t_{1}, t_{2}, -t_{3}<1\rbrace.\notag
\end{equation}
We use bold letters \textbf{t}, \textbf{s}, etc. to denote homogeneous coordinates. Also we denote the subset of $\mathbb{R}_{H}^{3}$ consists of points with integer components by $\mathbb{Z}_{H}^{3}.$

A function $f:\mathbb{R}^{2}\rightarrow \mathbb{C}$ is periodic with respect to the hexagonal lattice, or briefly $H-$periodic if
\begin{equation}
    f(x+Hk)=f(x),\hspace{3mm} k\in \mathbb{Z}^{2}.\notag
\end{equation}
If we define  
\begin{equation}
    \textbf{t}\equiv \textbf{s} \hspace{3mm}(\text{mod3})  \Longleftrightarrow t_{1}-s_{1}\equiv t_{2}-s_{2}\equiv t_{3}-s_{3}\hspace{3mm} (\text {mod3})\notag
\end{equation}
for $\textbf{t}=(t_{1}, t_{2}, t_{3}), \textbf{s}=(s_{1}, s_{2}, s_{3})\in \mathbb{R}_{H}^{3} $, it follows that the function $f$ is $H-$periodic if and only if 
\begin{equation}
    f(\textbf{t}+\textbf{s})=f(\textbf{t})\notag
\end{equation}
whenever $\textbf{s}\equiv \textbf{0}$ (mod3), in terms of homogeneous coordinates. Also, 
\begin{equation}\label{periodic}
\int_{\Omega}f(\textbf{t}+\textbf{s})d\textbf{t}=\int_{\Omega}f(\textbf{t})d\textbf{t},\hspace{2mm} \textbf{s}\in \mathbb{R}_{H}^{3}
\end{equation}
for the $H-$periodic function $f$ (\cite{xu}).

$L^{2}({\Omega})$ is a Hilbert space with respect to the inner product
\begin{equation}
    \langle f,g\rangle_{H}:=\frac{1}{\left |\Omega \right |}\int_{\Omega}f(\textbf{t})\overline{g(\textbf{t})}d\textbf{t},\notag
\end{equation}
where $\left |\Omega \right |$ denotes the area of $\Omega.$ The functions
\begin{equation}
    \phi_{\textbf{j}}(\textbf{t}):=e^{\frac{2\pi i}{3}\langle \textbf{j},\textbf{t}\rangle},\hspace{2mm} \textbf{j}\in \mathbb{Z}^{3}_{H}, \textbf{t}\in \mathbb{R}^{3}_{H},\notag
\end{equation}
where $\langle \textbf{j},\textbf{t}\rangle$ is the usual Euclidean inner product of  \textbf{j} and \textbf{t}, are $H-$periodic, and by Theorem 1, the set
\begin{equation}
    \lbrace  \phi_{\textbf{j}}:\textbf{j}\in \mathbb{Z}^{3}_{H}\rbrace\notag
\end{equation}
becomes an orthonormal basis for $L^{2}({\Omega}).$  See \cite{li sun xu} and \cite{xu} for details.

For every natural number $n$ we define the set
\begin{equation}
    \mathbb{H}_{n}:=\lbrace \textbf{j}=(j_{1},j_{2},j_{3})\in \mathbb{Z}^{3}_{H}:-n\leq j_{1},j_{2},j_{3}\leq n\rbrace, \notag
\end{equation}
which consists of points inside the hexagon $n\overline{\Omega}$ with integer components.

The hexagonal Fourier series of an $H-$periodic function $f\in L^{2}(\Omega)$ is defined by
\begin{equation}\label{series}
    f(\textbf{t})\sim \sum_{\textbf{j}\in \mathbb{Z}^{3}_{H}}\widehat{f}_{\textbf{j}}\phi_{\textbf{j}}(\textbf{t}),
\end{equation}
where 
\begin{equation}
    \widehat{f}_{\textbf{j}}=\frac{1}{\left |\Omega \right |}\int_{\Omega}f(\textbf{t})\overline{\phi_{\textbf{j}}(\textbf{t})}d\textbf{t},\hspace{3mm}\textbf{j}\in \mathbb{Z}^{3}_{H}.\notag
\end{equation}
The $n$-th hexagonal partial sum of (\ref{series}) is defined by 
\begin{equation}\label{partial}
    S_{n}(f)(\textbf{t}):=\sum_{\textbf{j}\in \mathbb{H}_{n}}\widehat{f}_{\textbf{j}}\phi_{\textbf{j}}(\textbf{t})=\frac{1}{\left |\Omega \right |}\int_{\Omega}f(\textbf{t}-\textbf{s})D_{n}(\textbf{s})d\textbf{s},
\end{equation}
where the second equality of (\ref{partial}) follows from (\ref{periodic}) with the kernel $D_{n}$ defined by
\begin{equation}
    D_{n}(\textbf{t}):=\sum_{\textbf{j}\in \mathbb{H}_{n}}\phi_{\textbf{j}}(\textbf{t}).\notag
\end{equation}
The kernel $D_{n}$ is an analogue of the Dirichlet kernel for the classical Fourier series. It is known by \cite{li sun xu} (see also \cite{xu}) that
\begin{equation}\label{dirichlet}
    D_{n}(\textbf{t})=\Theta_{n}(\textbf{t})-\Theta_{n-1}(\textbf{t}),
\end{equation}
where
\begin{equation}\label{theta}
    \Theta_{n}(\textbf{t}):=\frac{\sin \frac{(n+1)(t_{1}-t_{2})\pi}{3}\sin \frac{(n+1)(t_{2}-t_{3})\pi}{3}\sin \frac{(n+1)(t_{3}-t_{1})\pi}{3}}{\sin \frac{(t_{1}-t_{2})\pi}{3}\sin \frac{(t_{2}-t_{3})\pi}{3}\sin \frac{(t_{3}-t_{1})\pi}{3}}
\end{equation}
for $\textbf{t}=(t_{1},t_{2},t_{3})\in \mathbb{R}^{3}_{H}.$

The degree of approximation by various means of hexagonal Fourier series was studied in \cite{guven jca}, \cite{guven mia}, \cite{guven cmft}, \cite{guven zaa}, \cite{guven jnaat 2018}, \cite{guven rima}, \cite{guven jnaat 2020} and \cite{guven tjm}. In \cite{guven jca}, \cite{guven mia} and \cite{guven cmft}, some estimates for the degree of approximation of $(C,1)$ and Abel-Poisson means of hexagonal Fourier series were obtained. But, in these works there are no estimates in terms of the moduli of continuity of functions. In this article the degree of approximation of $(C,1)$ and Abel-Poisson means of hexagonal series of a function belong to $C(\overline{\Omega})$ is estimated directly in terms of modulus of continuity of the function.

\section{Approximation by ($C,\delta$) Means on Hexagonal Domain}

For $\delta > 0$ we denote the Ces\`{a}ro $(C,\delta)$ means of the series (\ref{series}) by $S^{(\delta)}_{n}(f)$, i. e.

\begin{equation}
    \ S^{(\delta)}_{n}(f)(\textbf{t})=\frac{1}{A^{\delta}_{n}}\sum_{k=0}^{n}A^{\delta-1}_{n-k}S_{k}(f)(\textbf{t}),\hspace{3mm}A^{\delta}_{n}=\begin{pmatrix}
        n+\delta \\
        \delta
    \end{pmatrix}.\notag
\end{equation}
It is easy to show that
\begin{equation}\label{cesarodelta}
    S^{(\delta)}_{n}(f)(\textbf{t})=\frac{1}{\left |\Omega\right |}\int_{\Omega}f(\textbf{t}-\textbf{s})K^{(\delta)}_{n}(\textbf{s})d\textbf{s},
\end{equation}
where
\begin{equation}
    K^{(\delta)}_{n}(\textbf{t})=\frac{1}{A^{\delta}_{n}}\sum_{k=0}^{n}A^{\delta-1}_{n-k}D_{k}(\textbf{t}).\notag
\end{equation}
If we set $\Theta_{-1}(\textbf{t}):=0$, we can express the kernel $K^{(\delta)}_{n}$ as
\begin{equation}\label{thetadelta}
    K^{(\delta)}_{n}(\textbf{t})=\frac{1}{A^{\delta}_{n}}\sum_{k=0}^{n}A^{\delta-1}_{n-k}\left (\Theta_{k}(\textbf{t})-\Theta_{k-1}(\textbf{t})\right )
\end{equation}
by aim of (\ref{dirichlet}).
Since 
\begin{equation}
    \frac{1}{\left |\Omega\right |}\int_{\Omega}D_{k}(\textbf{t})d\textbf{t}=1\notag
\end{equation}
and
\begin{equation}
    A^{\delta}_{n}=\sum_{k=0}^{n}A^{\delta-1}_{k}\notag
\end{equation}
we have
\begin{equation}\label{kernel}
    \frac{1}{\left |\Omega\right |}\int_{\Omega}K^{(\delta)}_{n}(\textbf{t})d\textbf{t}=1.
\end{equation}

Let $C(\overline {\Omega})$ be the Banach space of $H-$periodic continuous functions $f:\mathbb{R}^{3}_{H}\rightarrow \mathbb{C}$, whose norm is the uniform norm:
\begin{equation}
    \left \|f\right \|_{C(\overline {\Omega})}=\sup_{\textbf{t}\in \overline {\Omega}}\left |f(\textbf{t})\right |.\notag
\end{equation}
The modulus of continuity of a function $f\in C(\overline {\Omega})$ is defined by
\begin{equation}
    \omega_{f} (u):=\sup_{0<\left \|\textbf{t}\right \|\leq u}\left \|f-f(\cdot +\textbf{t})\right \|_{C(\overline {\Omega})},\hspace{3mm}u>0\notag
\end{equation}
where
\begin{equation}
    \left \|\textbf{t}\right \|=\max \lbrace \left |t_{1}\right |,\left |t_{2}\right |, \left |t_{3}\right |\rbrace,\hspace{3mm} \textbf{t}=(t_{1},t_{2},t_{3})\in \mathbb{R}^{3}_{H}.\notag
\end{equation}
$\omega_{f} (\cdot)$ is a non-negative and non-decreasing function which satisfies 
\begin{equation}\label{mod}
    \omega_{f} (\lambda u)\leq (1+\lambda)\omega_{f} (u)
\end{equation}
for $\lambda >0$ (\cite{xu}).

In the rest of the paper we shall write $A\lesssim B$ for the quantities $A$ and $B$ if there exists a constant $K>0$ ($K$ is an absolute constant, or a constant depending only on parameters which are not important for the questions involve in the paper) such that $A\leq KB$ holds.
\begin{lemma}
    The inequality
    \begin{equation}\label{cos}
        \left | \frac{1}{A^{\delta}_{n}}\sum_{k=0}^{n}A^{\delta-1}_{n-k}\cos (2k+1 )t    \right |\lesssim \frac{1}{(n+1)^{\delta}(\sin t)^{\delta}}+\frac{1}{(n+1)\sin t}
    \end{equation}
    holds for $0<\delta <1, 0<t<\pi$ and for every natural number $n.$
\end{lemma}
\begin{proof}
    It is known that (\cite[page 94]{zygmund})
    \begin{equation}
     \left | \frac{1}{A^{\delta}_{n}\sin \frac{t}{2}} \sum_{k=0}^{n}A^{\delta-1}_{n-k}e^{i\left ((2k+1 )\frac{t}{2}\right )} \right |\lesssim \frac{1}{A^{\delta}_{n}\left (2\sin \frac{t}{2}\right )^{\delta +1}} +\frac{2A^{\delta -1}_{n+1}}{A^{\delta}_{n}\left (2\sin \frac{t}{2}\right )^{2}} \notag
    \end{equation}
    for $0<t<2\pi.$ Hence, we have for $0<t<2\pi$
    \begin{equation}
   \left | \frac{1}{A^{\delta}_{n}\sin \frac{t}{2}}\sum_{k=0}^{n}A^{\delta-1}_{n-k}\cos (2k+1)\frac{t}{2}\right |\lesssim \frac{1}{A^{\delta}_{n}\left (2\sin \frac{t}{2}\right )^{\delta +1}} +\frac{2A^{\delta -1}_{n+1}}{A^{\delta}_{n}\left (2\sin \frac{t}{2}\right )^{2}}\notag
    \end{equation}
  since $\left |\Re (z)\right |\leq \left |z\right |$ for every complex number $z$. This implies
  \begin{equation}
    \left | \frac{1}{A^{\delta}_{n}}\sum_{k=0}^{n}A^{\delta-1}_{n-k}\cos (2k+1)t\right |\lesssim \frac{1}{A^{\delta}_{n}\left (\sin t\right )^{\delta}} +\frac{A^{\delta -1}_{n+1}}{A^{\delta}_{n}\sin t}\hspace{3mm} (0<t<\pi ).\notag
  \end{equation}
  Since 
  \begin{equation}
      \frac{A^{\delta -1}_{n+1}}{A^{\delta}_{n}}=\frac{\delta}{n+1}\notag
  \end{equation}
  we get
  \begin{equation}
    \left | \frac{1}{A^{\delta}_{n}}\sum_{k=0}^{n}A^{\delta-1}_{n-k}\cos (2k+1)t\right |\lesssim \frac{1}{A^{\delta}_{n}\left (\sin t\right )^{\delta}} +\frac{1}{(n+1)\sin t}.\notag
  \end{equation}
  This inequality and the fact $A^{\delta}_{n}\approx (n+1)^{\delta}$ yields (\ref{cos}).
\end{proof}

\begin{theorem}
For $0<\delta <1$ the kernel $K^{(\delta)}_{n}$ satisfies the inequality
\begin{equation}\label{Kndelta}
    \int_{\Omega}\left |K^{(\delta)}_{n}(\textbf{t})\right |d\textbf{t}\lesssim \log (n+1).
\end{equation}
\end{theorem}
\begin{proof}
\begin{equation}
    \int_{\Omega}\left |K^{(\delta)}_{n}(\textbf{t})\right |d\textbf{t}=\frac{1}{ A^{\delta}_{n}} \int_{\Omega}\left |\sum_{k=0}^{n}A^{\delta-1}_{n-k}\left (\Theta_{k}(\textbf{t})-\Theta_{k-1}(\textbf{t})\right ) \right |d\textbf{t}.\notag
\end{equation}
Since the function
\begin{equation}
    \textbf{t}\rightarrow \left |\sum_{k=0}^{n}A^{\delta-1}_{n-k}\left (\Theta_{k}(\textbf{t})-\Theta_{k-1}(\textbf{t})\right ) \right |\notag
\end{equation}
is symmetric with respect to variables $t_{1},t_{2}$ and $t_{3}$, where $\textbf{t}=(t_{1},t_{2},t_{3})\in \Omega$, it is sufficient to estimate the integral over the triangle
\begin{align}
    \Delta:&=\lbrace \textbf{t}=(t_{1},t_{2},t_{3})\in \mathbb{R}_{H}^{3}:0\leq t_{1},t_{2},-t_{3}\leq 1\rbrace \notag\\
    &=\lbrace (t_{1},t_{2}):t_{1}\geq 0,t_{2}\geq 0,t_{1}+t_{2}\leq 1\rbrace, \notag
\end{align}
which is one of the six equilateral triangles in $\overline{\Omega}$. By considering (\ref{theta}),
\begin{align}
    & \frac{1}{ A^{\delta}_{n}} \int_{\Delta}\left |\sum_{k=0}^{n}A^{\delta-1}_{n-k}\left (\Theta_{k}(\textbf{t})-\Theta_{k-1}(\textbf{t})\right ) \right |d\textbf{t}\notag\\
     &=\frac{1}{ A^{\delta}_{n}} \int_{\Delta}\left |\sum_{k=0}^{n}A^{\delta-1}_{n-k}
     \begin{pmatrix}
      \frac{\sin \frac{(k+1)(t_{1}-t_{2})\pi}{3}\sin \frac{(k+1)(t_{2}-t_{3})\pi}{3}\sin \frac{(k+1)(t_{3}-t_{1})\pi}{3}}{\sin \frac{(t_{1}-t_{2})\pi}{3}\sin \frac{(t_{2}-t_{3})\pi}{3}\sin \frac{(t_{3}-t_{1})\pi}{3}} \\-\frac{\sin \frac{k(t_{1}-t_{2})\pi}{3}\sin \frac{k(t_{2}-t_{3})\pi}{3}\sin \frac{(k(t_{3}-t_{1})\pi}{3}}{\sin \frac{(t_{1}-t_{2})\pi}{3}\sin \frac{(t_{2}-t_{3})\pi}{3}\sin \frac{(t_{3}-t_{1})\pi}{3}}
      \end{pmatrix}\right|d\textbf{t}.\notag
\end{align}
If we use the change of variables
\begin{equation}\label{transform1}
    s_{1}:=\frac{t_{1}-t_{3}}{3}=\frac{2t_{1}+t_{2}}{3},  s_{2}:=\frac{t_{2}-t_{3}}{3}=\frac{t_{1}+2t_{2}}{3}
\end{equation}
the integral becomes
\begin{equation}
    3\frac{1}{ A^{\delta}_{n}} \int_{\tilde {\Delta}}\left |\sum_{k=0}^{n}A^{\delta-1}_{n-k}
     \begin{pmatrix}
      \frac{\sin ((k+1)(s_{1}-s_{2})\pi)\sin ((k+1)s_{2}\pi)\sin ((k+1)(-s_{1})\pi)}{\sin ((s_{1}-s_{2})\pi)\sin (s_{2}\pi)\sin ((-s_{1})\pi)} \\-\frac{\sin (k(s_{1}-s_{2})\pi)\sin (k(s_{2}\pi)\sin (k(-s_{1})\pi)}{\sin ((s_{1}-s_{2})\pi)\sin (s_{2}\pi)\sin ((-s_{1})\pi)}
      \end{pmatrix}\right |ds_{1}ds_{2},\notag
\end{equation}
where $\tilde {\Delta}$ is the image of $\Delta$ in the plane, that is
\begin{equation}
    \tilde {\Delta}:=\lbrace (s_{1},s_{2}):0\leq s_{1}\leq 2s_{2},0\leq s_{2}\leq 2s_{1},s_{1}+s_{2}\leq 1\rbrace.\notag
\end{equation}
Since the integrated function is symmetric with respect to $s_{1}$ and $s_{2}$, estimating the integral over the triangle 
\begin{equation}
    \Delta^{*}:=\lbrace (s_{1},s_{2})\in  \tilde {\Delta} : s_{1}\leq s_{2}\rbrace =\lbrace (s_{1},s_{2}): s_{1}\leq s_{2}\leq 2s_{1}, s_{1}+s_{2}\leq 1\rbrace,\notag
\end{equation}
which is the half of $\tilde {\Delta}$, will be sufficient. The change of variables
\begin{equation}\label{transform2}
    s_{1}:=\frac{u_{1}-u_{2}}{2}, s_{2}:=\frac{u_{1}+u_{2}}{2}
\end{equation}
transforms the triangle $\Delta^{*}$ to the triangle 
\begin{equation}
    \Gamma :=\lbrace (u_{1},u_{2}):0\leq u_{2}\leq \frac{u_{1}}{3}, 0\leq u_{1}\leq 1\rbrace.\notag
\end{equation}
Thus we must estimate
\begin{equation}
    I_{n}^{(\delta)}:=\frac{1}{ A^{\delta}_{n}}\int_{\Gamma}\left |\sum_{k=0}^{n}A^{\delta-1}_{n-k}D_{k}^{*}(u_{1},u_{2})\right |du_{1}du_{2},\notag
\end{equation}
where
\begin{align}
    D_{k}^{*}(u_{1},u_{2}):=& \frac{\sin ((k+1)u_{2}\pi)\sin ((k+1)\frac{u_{1}+u_{2}}{2}\pi) \sin ((k+1)\frac{u_{1}-u_{2}}{2}\pi)}{\sin (u_{2}\pi)\sin (\frac{u_{1}+u_{2}}{2}\pi) \sin (\frac{u_{1}-u_{2}}{2}\pi)} \notag \\
    &-\frac{\sin (ku_{2}\pi)\sin (k\frac{u_{1}+u_{2}}{2}\pi) \sin (k\frac{u_{1}-u_{2}}{2}\pi)}{\sin (u_{2}\pi)\sin (\frac{u_{1}+u_{2}}{2}\pi) \sin (\frac{u_{1}-u_{2}}{2}\pi)}.\notag
\end{align}
Using elementary trigonometric identities yields
\begin{equation}
     D_{k}^{*}(u_{1},u_{2})= D_{k,1}^{*}(u_{1},u_{2})+ D_{k,2}^{*}(u_{1},u_{2})+ D_{k,3}^{*}(u_{1},u_{2}),\notag
\end{equation}
where
\begin{align}
    & D_{k,1}^{*}(u_{1},u_{2})=2\cos \left (\left (k+\frac{1}{2}\right )u_{2}\pi\right )\frac{\sin (\frac{1}{2}u_{2}\pi)\sin ((k+1)\frac{u_{1}+u_{2}}{2}\pi) \sin ((k+1)\frac{u_{1}-u_{2}}{2}\pi)}{\sin (u_{2}\pi)\sin (\frac{u_{1}+u_{2}}{2}\pi) \sin (\frac{u_{1}-u_{2}}{2}\pi)}, \notag\\
    &D_{k,2}^{*}(u_{1},u_{2})=2\cos \left (\left (k+\frac{1}{2}\right )\frac{u_{1}+u_{2}}{2}\pi\right )\frac{\sin (ku_{2}\pi)\sin (\frac{1}{2}\frac{u_{1}+u_{2}}{2}\pi) \sin ((k+1)\frac{u_{1}-u_{2}}{2}\pi)}{\sin (u_{2}\pi)\sin (\frac{u_{1}+u_{2}}{2}\pi) \sin (\frac{u_{1}-u_{2}}{2}\pi)}, \notag\\
    &D_{k,3}^{*}(u_{1},u_{2})=2\cos \left (\left (k+\frac{1}{2}\right )\frac{u_{1}-u_{2}}{2}\pi\right )\frac{\sin (ku_{2}\pi)\sin (k\frac{u_{1}+u_{2}}{2}\pi) \sin (\frac{1}{2}\frac{u_{1}-u_{2}}{2}\pi)}{\sin (u_{2}\pi)\sin (\frac{u_{1}+u_{2}}{2}\pi) \sin (\frac{u_{1}-u_{2}}{2}\pi)}. \notag
\end{align}
We can write $\Gamma = \Gamma_{1}\cup \Gamma_{2}\cup \Gamma_{3},$ where
\begin{align}
    &\Gamma_{1}:=\lbrace (u_{1},u_{2})\in \Gamma :u_{1}\leq \frac{1}{n+1}\rbrace,\notag\\
    &\Gamma_{2}:=\lbrace (u_{1},u_{2})\in \Gamma :u_{1}\geq \frac{1}{n+1}, u_{2}\leq \frac{1}{3(n+1)}\rbrace,\notag\\
    &\Gamma_{3}:=\lbrace (u_{1},u_{2})\in \Gamma :u_{1}\geq \frac{1}{n+1}, u_{2}\geq \frac{1}{3(n+1)}\rbrace.\notag
\end{align}
Hence we have $I_{n}^{(\delta)}=I_{n,1}^{(\delta)}+I_{n,2}^{(\delta)}+I_{n,3}^{(\delta)}$, where
\begin{equation}
    I_{n,j}^{(\delta)}:=\frac{1}{ A^{\delta}_{n}}\int_{\Gamma_{j}}\left |\sum_{k=0}^{n}A^{\delta-1}_{n-k}D_{k}^{*}(u_{1},u_{2})\right |du_{1}du_{2}\hspace{3mm}(j=1,2,3).\notag
\end{equation}
We shall use well known inequalities
\begin{equation}\label{sin}
    \left |\frac{\sin (nt)}{\sin t}\right |\leq n \hspace{3mm}(n\in \mathbb{N})
\end{equation}
and
\begin{equation}\label{jordan}
    \sin t\geq \frac{2}{\pi}t\hspace{3mm}\left (0\leq t\leq \frac{\pi}{2}\right )
\end{equation}
to estimate the integrals $I_{n,j}^{(\delta)}\hspace{1.5mm}(j=1,2,3).$
Simple calculations give 
\begin{equation}
    \frac{\sin \left (\frac{u_{2}\pi}{2}\right )}{\sin (u_{2}\pi)}\leq 1, \frac{\sin \left (\frac{1}{2}\frac{u_{1}+u_{2}}{2}\pi\right )}{\sin \left (\frac{u_{1}+u_{2}}{2}\pi\right )}\leq 1, \frac{\sin \left (\frac{1}{2}\frac{u_{1}-u_{2}}{2}\pi\right )}{\sin \left (\frac{u_{1}-u_{2}}{2}\pi\right )}\leq 1\notag
\end{equation}
for $(u_{1},u_{2})\in \Gamma.$  By considering these inequalities and (\ref{sin}) we get
\begin{equation}\label{dk}
    \left |D_{k,1}^{*}(u_{1},u_{2})\right |\leq (k+1)^{2},  \left |D_{k,2}^{*}(u_{1},u_{2})\right |\leq k(k+1),  \left |D_{k,3}^{*}(u_{1},u_{2})\right |\leq k^{2}
\end{equation}
for $(u_{1},u_{2})\in \Gamma_{1}.$ Hence,
\begin{align}
     I_{n,1}^{(\delta)}&\leq\frac{1}{ A^{\delta}_{n}}\sum_{k=0}^{n}A^{\delta-1}_{n-k}(k+1)^{2}\int_{\Gamma_{1}}du_{1}du_{2}\notag\\
     &=\frac{1}{ A^{\delta}_{n}}\sum_{k=0}^{n}A^{\delta-1}_{n-k}(k+1)^{2}\frac{1}{6(n+1)^{2}}\leq \frac{1}{6}\leq 1,
\end{align}
since (\cite[page 77]{zygmund})
\begin{equation}
\sum_{k=0}^{n}A^{\delta-1}_{n-k}=A^{\delta}_{n}.
\end{equation}
We can write the rectangle $\Gamma_{2}$ as the union of rectangles 
\begin{align}
    &\Gamma_{2}^{'}=\lbrace (u_{1},u_{2})\in \Gamma_{2}:u_{2}\leq \frac{1}{(n+1)^{2}}\rbrace\notag \\
    &\Gamma_{2}^{''}=\lbrace (u_{1},u_{2})\in \Gamma_{2}:u_{2}\geq \frac{1}{(n+1)^{2}}\rbrace.\notag 
\end{align}
It is easy to see that 
\begin{equation}\label{sinsin}
    \sin \left (\frac{u_{1}+u_{2}}{2}\pi\right )\geq \frac{\sqrt{3}}{2}\sin \left (\frac{u_{1}\pi}{2}\right ) \hspace{3mm}\text{and}\hspace{3mm}\sin \left (\frac{u_{1}-u_{2}}{2}\pi\right )\geq \sin \left (\frac{u_{1}\pi}{3}\right ) 
\end{equation}
for $(u_{1},u_{2})\in \Gamma.$ Considering these inequalities and (\ref{jordan}) yields
\begin{equation}\label{dk2}
    \left |D_{k,1}^{*}(u_{1},u_{2})\right |\lesssim \frac{1}{u_{1}^{2}} \hspace{3mm}\text{and}\hspace{3mm} \left |D_{k,j}^{*}(u_{1},u_{2})\right |\lesssim \frac{k}{u_{1}} \hspace{3mm}(j=2,3)
\end{equation}
for $(u_{1},u_{2})\in \Gamma_{2}^{'}.$ Hence
\begin{align}
    &\frac{1}{ A^{\delta}_{n}}\int_{\Gamma_{2}^{'}}\left |\sum_{k=0}^{n}A^{\delta-1}_{n-k}D_{k,1}^{*}(u_{1},u_{2})\right |du_{1}du_{2}\notag \\
    &\leq \frac{1}{ A^{\delta}_{n}}\sum_{k=0}^{n}A^{\delta-1}_{n-k}\left (\int_{\Gamma_{2}^{'}}\left |D_{k,1}^{*}(u_{1},u_{2})\right |du_{1}du_{2}\right)\notag\\
    &\lesssim \frac{1}{ A^{\delta}_{n}}\sum_{k=0}^{n}A^{\delta-1}_{n-k}\left (\int_{\Gamma_{2}^{'}}\frac{1}{u_{1}^{2}}du_{1}du_{2}\right)\notag\\
    &=\frac{1}{ A^{\delta}_{n}}\sum_{k=0}^{n}A^{\delta-1}_{n-k}\left (\int_{0}^{\frac{1}{3(n+1)^2}} \int_{\frac{1}{n+1}}^{1}\frac{1}{u_{1}^{2}}du_{1}du_{2}\right)\notag\\
    &\leq \frac{1}{n+1},\notag
\end{align}
and for $j=2,3$
\begin{align}
     &\frac{1}{ A^{\delta}_{n}}\int_{\Gamma_{2}^{'}}\left |\sum_{k=0}^{n}A^{\delta-1}_{n-k}D_{k,j}^{*}(u_{1},u_{2})\right |du_{1}du_{2}\notag \\
     &\lesssim \frac{1}{ A^{\delta}_{n}}\sum_{k=0}^{n}A^{\delta-1}_{n-k}k\left (\int_{\Gamma_{2}^{'}}\frac{1}{u_{1}}du_{1}du_{2}\right)\notag\\
     &\leq \frac{n}{ A^{\delta}_{n}}\sum_{k=0}^{n}A^{\delta-1}_{n-k}\left (\int_{0}^{\frac{1}{3(n+1)^2}} \int_{\frac{1}{n+1}}^{1}\frac{1}{u_{1}}du_{1}du_{2}\right)\notag\\
     &=\frac{n}{3(n+1)^{2}}\log (n+1)\notag \\
     &\leq \frac{\log (n+1)}{n+1}.\notag
\end{align}
For $(u_{1},u_{2})\in \Gamma_{2}^{''}\cup \Gamma_{3}$ we need another expression of the function $D_{k}^{*}.$ Since
\begin{equation}
    \sin (2x)+\sin (2y)+\sin (2z)=-4\sin x\sin y\sin z\notag
\end{equation}
for $x+y+z=0$, we have
\begin{equation}
    D_{k}^{*}(u_{1},u_{2})=H_{k,1}^{*}(u_{1},u_{2})+H_{k,2}^{*}(u_{1},u_{2})+H_{k,3}^{*}(u_{1},u_{2}),\notag
\end{equation}
where
\begin{align}
    &H_{k,1}^{*}(u_{1},u_{2}):=\frac{1}{2}\frac{\cos ((2k+1)u_{2}\pi)}{\sin \left(\frac{u_{1}+u_{2}}{2}\pi\right ) \sin \left (\frac{u_{1}-u_{2}}{2}\pi\right)}\notag\\
    &H_{k,2}^{*}(u_{1},u_{2}):=-\frac{1}{2}\frac{\cos \left ((2k+1)\frac{u_{1}+u_{2}}{2}\pi\right)}{\sin (u_{2}\pi ) \sin \left (\frac{u_{1}-u_{2}}{2}\pi\right)}\notag\\
    &H_{k,3}^{*}(u_{1},u_{2}):=\frac{1}{2}\frac{\cos \left ((2k+1)\frac{u_{1}-u_{2}}{2}\pi\right)}{\sin (u_{2}\pi ) \sin \left (\frac{u_{1}+u_{2}}{2}\pi\right)}.\notag
\end{align}
By (\ref{jordan}) and (\ref{cos}) we have
\begin{align}
   \left |\frac{1}{ A^{\delta}_{n}} \sum_{k=0}^{n}A^{\delta-1}_{n-k}H_{k,1}^{*}(u_{1},u_{2})\right | &\lesssim \frac{1}{u_{1}^{2}}  \left |\frac{1}{ A^{\delta}_{n}} \sum_{k=0}^{n}A^{\delta-1}_{n-k} \cos ((2k+1)u_{2}\pi)\right |\notag\\
   &\lesssim \frac{1}{u_{1}^{2}}\left (\frac{1}{(n+1)^{\delta}(\sin (u_{2}\pi))^{\delta}}+\frac{1}{(n+1)\sin(u_{2}\pi)}\right )\notag\\
   &\leq \frac{1}{u_{1}^{2}}\left ( \frac{1}{(n+1)^{\delta}u_{2}^{\delta}}+\frac{1}{(n+1)u_{2}}\right ).\notag
\end{align}
Since $(n+1)u_{2}<1$ for $(u_{1},u_{2})\in \Gamma_{2}^{''}$, 
\begin{equation}\label{Hk0}
 \left |\frac{1}{ A^{\delta}_{n}} \sum_{k=0}^{n}A^{\delta-1}_{n-k}H_{k,1}^{*}(u_{1},u_{2})\right | \lesssim \frac{1}{u_{1}^{2}}\frac{1}{(n+1)u_{2}}, \hspace{3mm}(u_{1},u_{2})\in \Gamma_{2}^{''},
\end{equation}
and since  $3(n+1)u_{2}\geq 1$ for $(u_{1},u_{2})\in \Gamma_{3}$, we get
\begin{equation}\label{Hk1}
 \left |\frac{1}{ A^{\delta}_{n}} \sum_{k=0}^{n}A^{\delta-1}_{n-k}H_{k,1}^{*}(u_{1},u_{2})\right | \lesssim \frac{1}{u_{1}^{2}}\frac{1}{(n+1)^{\delta}u_{2}^{\delta}}, \hspace{3mm}(u_{1},u_{2})\in \Gamma_{3}.
\end{equation}
By similar way we obtain
\begin{equation}\label{Hk23}
 \left |\frac{1}{ A^{\delta}_{n}} \sum_{k=0}^{n}A^{\delta-1}_{n-k}H_{k,j}^{*}(u_{1},u_{2})\right | \lesssim \frac{1}{u_{1}^{\delta +1}}\frac{1}{(n+1)^{\delta}u_{2}}, \hspace{3mm}(u_{1},u_{2})\in \Gamma_{2}^{''}\cup \Gamma_{3},
\end{equation}
for $j=2,3.$
By (\ref{Hk1}),
\begin{align}
    \frac{1}{A_{n}^{\delta}}\int_{\Gamma_{2}^{''}}\left |\sum_{k=0}^{n}A^{\delta-1}_{n-k}H_{k,1}^{*}(u_{1},u_{2})\right |du_{1}du_{2}&=\int_{\Gamma_{2}^{''}}\left |\frac{1}{A_{n}^{\delta}}\sum_{k=0}^{n}A^{\delta-1}_{n-k}H_{k,1}^{*}(u_{1},u_{2})\right |du_{1}du_{2}\notag\\
    &\lesssim \frac{1}{n+1}\int_{\frac{1}{3(n+1)^{2}}}^{\frac{1}{3(n+1)}}\int_{\frac{1}{n+1}}^{1}\frac{1}{u_{1}^{2}}\frac{1}{u_{2}}du_{1}du_{2}\notag\\
    &=\frac{n}{n+1}\log (n+1)\leq \log (n+1).\notag
\end{align}
By (\ref{Hk23}) we get
\begin{align}
     \frac{1}{A_{n}^{\delta}}\int_{\Gamma_{2}^{''}}\left |\sum_{k=0}^{n}A^{\delta-1}_{n-k}H_{k,j}^{*}(u_{1},u_{2})\right |du_{1}du_{2}&=\int_{\Gamma_{2}^{''}}\left |\frac{1}{A_{n}^{\delta}}\sum_{k=0}^{n}A^{\delta-1}_{n-k}H_{k,j}^{*}(u_{1},u_{2})\right |du_{1}du_{2}\notag\\
     &\lesssim \frac{1}{(n+1)^{\delta}}\int_{\frac{1}{3(n+1)^{2}}}^{\frac{1}{3(n+1)}}\int_{\frac{1}{n+1}}^{1}\frac{1}{u_{1}^{\delta +1}}\frac{1}{u_{2}}du_{1}du_{2}\notag\\
     &=\frac{\log (n+1)}{(n+1)^{\delta}}\left ((n+1)^{\delta}-1\right )\notag \\
     &\leq \log (n+1)\notag
\end{align}
for $j=2,3.$ Combining these estimates yields
\begin{equation}
    I_{n,2}^{(\delta)}\lesssim \log (n+1).\notag
\end{equation}
By (\ref{Hk1})
\begin{align}
     \frac{1}{A_{n}^{\delta}}\int_{\Gamma_{3}}\left |\sum_{k=0}^{n}A^{\delta-1}_{n-k}H_{k,1}^{*}(u_{1},u_{2})\right |du_{1}du_{2}&=\int_{\Gamma_{3}}\left |\frac{1}{A_{n}^{\delta}}\sum_{k=0}^{n}A^{\delta-1}_{n-k}H_{k,1}^{*}(u_{1},u_{2})\right |du_{1}du_{2}\notag\\
     &\lesssim \frac{1}{(n+1)^{\delta}}\int_{\frac{1}{3(n+1)}}^{\frac{1}{3}}\int_{3u_{2}}^{1}\frac{1}{u_{1}^{2}}\frac{1}{u_{2}^{\delta}}du_{1}du_{2}\notag\\
     &=\frac{1}{(n+1)^{\delta}}\frac{3^{\delta}}{\delta}\left ((n+1)^{\delta}-1\right )\lesssim 1,\notag
\end{align}
and for $j=2,3,$ (\ref{Hk23}) gives
\begin{align}
    \frac{1}{A_{n}^{\delta}}\int_{\Gamma_{3}}\left |\sum_{k=0}^{n}A^{\delta-1}_{n-k}H_{k,j}^{*}(u_{1},u_{2})\right |du_{1}du_{2}&=\int_{\Gamma_{3}}\left |\frac{1}{A_{n}^{\delta}}\sum_{k=0}^{n}A^{\delta-1}_{n-k}H_{k,j}^{*}(u_{1},u_{2})\right |du_{1}du_{2}\notag\\
    &\lesssim \frac{1}{(n+1)^{\delta}}\int_{\frac{1}{n+1}}^{1}\int_{\frac{1}{3(n+1)}}^{\frac{u_{1}}{3}}\frac{1}{u_{1}^{\delta +1}}\frac{1}{u_{2}}du_{2}du_{1}\notag \\
    &=\frac{1}{A_{n}^{\delta}}\int_{\frac{1}{n+1}}^{1}\log \left ((n+1)u_{1}\right )\frac{1}{u_{1}^{\delta +1}}du_{1}\notag \\
    &\leq \frac{\log (n+1)}{(n+1)^{\delta}}\int_{\frac{1}{n+1}}^{1}\frac{1}{u_{1}^{\delta +1}}du_{1}\notag\\
    &=\frac{\log (n+1)}{(n+1)^{\delta}}\frac{1}{\delta}\left ((n+1)^{\delta}-1\right )\lesssim \log (n+1).\notag
\end{align}
Thus we have
\begin{equation}
     I_{n,3}^{(\delta)}\lesssim \log (n+1),\notag
\end{equation}
and hence (\ref{Kndelta}) follows.
\end{proof}

\begin{theorem}
    The estimate
    \begin{equation}\label{dndelta}
        \int_{\Omega}\left \|\textbf{t}\right \|\left |K^{(\delta)}_{n}(\textbf{t})\right |d\textbf{t}\lesssim
        \frac{\log (n+1)}{(n+1)^{\delta}}
    \end{equation}
    holds for $0<\delta <1$.
\end{theorem}
\begin{proof}
    As in proof of Theorem 2, it is sufficient to estimate the integral on the triangle $\Delta.$ By transforms (\ref{transform1}) and (\ref{transform2}) estimating the integral
    \begin{equation}
         \int_{\Delta}\left \|\textbf{t}\right \|\left |K^{(\delta)}_{n}(\textbf{t})\right |d\textbf{t}\notag
  \end{equation}
  is equivalent to estimating the integral
  \begin{equation}
      \int_{\Gamma}u_{1}\left |\frac{1}{A_{n}^{\delta}}\sum_{k=0}^{n}A_{n-k}^{\delta -1}D_{k}^{*}(u_{1},u_{2})\right |du_{1}du_{2}.\notag
  \end{equation}
  We can write
  \begin{equation}
      \int_{\Gamma}u_{1}\left |\frac{1}{A_{n}^{\delta}}\sum_{k=0}^{n}A_{n-k}^{\delta -1}D_{k}^{*}(u_{1},u_{2})\right |du_{1}du_{2}=J_{n,1}^{(\delta)}+J_{n,2}^{(\delta)}+J_{n,3}^{8\delta)},\notag
  \end{equation}
  where
  \begin{equation}
    J_{n,j}^{(\delta)}:=\int_{\Gamma_{j}}u_{1}\left |\frac{1}{A_{n}^{\delta}}\sum_{k=0}^{n}A_{n-k}^{\delta -1}D_{k}^{*}(u_{1},u_{2})\right |du_{1}du_{2},\hspace{3mm}j=1,2,3. \notag
  \end{equation}
  Considering (\ref{dk}) yields
  \begin{align}
      J_{n,1}^{(\delta)}&=\frac{1}{A_{n}^{\delta}}\int_{\Gamma_{1}}u_{1}\left |\sum_{k=0}^{n}A_{n-k}^{\delta -1}D_{k}^{*}(u_{1},u_{2})\right |du_{1}du_{2}\notag \\
      &\lesssim \int_{\Gamma_{1}}u_{1}(n+1)^{2}du_{1}du_{2}=(n+1)^{2}\int_{0}^{\frac{1}{3(n+1)}}\int_{3u_{2}}^{\frac{1}{n+1}} u_{1}du_{1}du_{2}\notag \\
      &=(n+1)^{2}\int_{0}^{\frac{1}{3(n+1)}}\left (\frac{1}{(n+1)^{2}}-(3u_{2})^{2}\right )du_{2}\leq \frac{1}{n+1}.\notag
\end{align}
By (\ref{dk2}),
\begin{align}
    \int_{\Gamma_{2}^{'}}u_{1}\left |\frac{1}{A_{n}^{\delta}}\sum_{k=0}^{n}A_{n-k}^{\delta -1}D_{k,1}^{*}(u_{1},u_{2})\right |du_{1}du_{2}&=\frac{1}{A_{n}^{\delta}}\int_{\Gamma_{2}^{'}}u_{1}\left |\sum_{k=0}^{n}A_{n-k}^{\delta -1}D_{k,1}^{*}(u_{1},u_{2})\right |du_{1}du_{2}\notag \\
    &\lesssim  \int_{\Gamma_{2}^{'}}u_{1}\frac{1}{u_{1}^{2}}du_{1}du_{2}\notag \\
    &=\int_{0}^{\frac{1}{3(n+1)^{2}}}\int_{\frac{1}{n+1}}^{1}\frac{1}{u_{1}}du_{1}du_{2}\notag \\
    &=\log (n+1)\int_{0}^{\frac{1}{3(n+1)^{2}}}du_{2}\notag \\
    &\leq \frac{\log (n+1)}{(n+1)^{2}},\notag
\end{align}
and for $j=2,3,$
\begin{align}
      \int_{\Gamma_{2}^{'}}u_{1}\left |\frac{1}{A_{n}^{\delta}}\sum_{k=0}^{n}A_{n-k}^{\delta -1}D_{k,j}^{*}(u_{1},u_{2})\right |du_{1}du_{2}&=\frac{1}{A_{n}^{\delta}}\int_{\Gamma_{2}^{'}}u_{1}\left |\sum_{k=0}^{n}A_{n-k}^{\delta -1}D_{k,j}^{*}(u_{1},u_{2})\right |du_{1}du_{2}\notag \\
      &\lesssim n\int_{\Gamma_{2}^{'}}u_{1}\frac{1}{u_{1}}du_{1}du_{2}\notag \\
      &=n\int_{0}^{\frac{1}{3(n+1)^{2}}}\int_{\frac{1}{n+1}}^{1}du_{1}du_{2}\leq \frac{1}{n+1}.\notag 
\end{align}
By (\ref{Hk0}),
\begin{align}
     \int_{\Gamma_{2}^{''}}u_{1}\left |\frac{1}{A_{n}^{\delta}}\sum_{k=0}^{n}A_{n-k}^{\delta -1}H_{k,1}^{*}(u_{1},u_{2})\right |du_{1}du_{2}&\lesssim \frac{1}{n+1}\int_{\Gamma_{2}^{''}}u_{1}\frac{1}{u_{1}^{2}u_{2}}du_{1}du_{2} \notag \\
     &=\frac{1}{n+1}\int_{\frac{1}{3(n+1)^{2}}}^{\frac{1}{3(n+1)}}\int_{\frac{1}{n+1}}^{1}\frac{1}{u_{1}u_{2}}du_{1}du_{2}\notag \\
     &=\frac{\left (\log (n+1)\right )^{2}}{n+1},\notag
\end{align}
and by (\ref{Hk23}),
\begin{align}
     \int_{\Gamma_{2}^{''}}u_{1}\left |\frac{1}{A_{n}^{\delta}}\sum_{k=0}^{n}A_{n-k}^{\delta -1}H_{k,j}^{*}(u_{1},u_{2})\right |du_{1}du_{2}&\lesssim \frac{1}{(n+1)^{\delta}}\int_{\Gamma_{2}^{''}}u_{1}\frac{1}{u_{1}^{\delta +1}u_{2}}du_{1}du_{2} \notag \\
     &=\frac{1}{(n+1)^{\delta}}\int_{\frac{1}{3(n+1)^{2}}}^{\frac{1}{3(n+1)}}\int_{\frac{1}{n+1}}^{1}\frac{1}{u_{1}^{\delta}u_{2}}du_{1}du_{2}\notag \\
     &=\frac{\log (n+1)}{(n+1)^{\delta}}\int_{\frac{1}{n+1}}^{1}\frac{1}{u_{1}^{\delta}}du_{1}\notag\\
     &=\frac{\log (n+1)}{(n+1)^{\delta}}\frac{1}{1-\delta}\left (1-\frac{1}{(n+1)^{1-\delta}}\right )\notag \\
     &\lesssim \frac{\log (n+1)}{(n+1)^{\delta}}.\notag
\end{align}
Since
\begin{equation}
    \frac{\left (\log (n+1)\right )^{2}}{n+1}\lesssim \frac{\log (n+1)}{(n+1)^{\delta}}\notag
\end{equation}
we get
\begin{equation}
    J_{n,2}^{(\delta)}\lesssim \frac{\log (n+1)}{(n+1)^{\delta}}.\notag
\end{equation}
By (\ref{Hk1}),
\begin{align}
\int_{\Gamma_{3}}u_{1}\left |\frac{1}{A_{n}^{\delta}}\sum_{k=0}^{n}A_{n-k}^{\delta -1}H_{k,1}^{*}(u_{1},u_{2})\right |du_{1}du_{2}&\lesssim \frac{1}{(n+1)^{\delta}}\int_{\Gamma_{3}}u_{1}\frac{1}{u_{1}^{2}u_{2}^{\delta}}du_{1}du_{2}\notag \\
&=\frac{1}{(n+1)^{\delta}}\int_{\frac{1}{3(n+1)}}^{\frac{1}{3}}\int_{3u_{2}}^{1}\frac{1}{u_{1}u_{2}^{\delta}}du_{1}du_{2}\notag \\
&=\frac{1}{(n+1)^{\delta}}\int_{\frac{1}{3(n+1)}}^{\frac{1}{3}}\log \left (\frac{1}{3u_{2}}\right )\frac{1}{u_{2}^{\delta}}du_{2}\notag \\
&\leq \frac{\log (n+1)}{(n+1)^{\delta}}\int_{\frac{1}{3(n+1)}}^{\frac{1}{3}}\frac{1}{u_{2}^{\delta}}du_{2}\notag \\
&\lesssim \frac{\log (n+1)}{(n+1)^{\delta}},\notag
\end{align}
and by (\ref{Hk23})
\begin{align}
    \int_{\Gamma_{3}}u_{1}\left |\frac{1}{A_{n}^{\delta}}\sum_{k=0}^{n}A_{n-k}^{\delta -1}H_{k,j}^{*}(u_{1},u_{2})\right |du_{1}du_{2}&\lesssim \frac{1}{(n+1)^{\delta}}\int_{\Gamma_{3}}u_{1}\frac{1}{u_{1}^{\delta +1}u_{2}}du_{1}du_{2}\notag \\
    &=\frac{1}{(n+1)^{\delta}}\int_{\frac{1}{n+1}}^{1}\int_{\frac{1}{3(n+1)}}^{\frac{u_{1}}{3}}\frac{1}{u_{1}^{\delta}u_{2}}du_{2}du_{1}\notag \\
    &=\frac{1}{(n+1)^{\delta}}\int_{\frac{1}{n+1}}^{1}\log \left ((n+1)u_{1}\right )\frac{1}{u_{1}^{\delta}}du_{1}\notag \\
    &\leq \frac{\log (n+1)}{(n+1)^{\delta}}\int_{\frac{1}{n+1}}^{1}\frac{1}{u_{1}^{\delta}}du_{1}\notag \\
    &\lesssim \frac{\log (n+1)}{(n+1)^{\delta}}.\notag
\end{align}
Hence we get
\begin{equation}
    J_{n,3}^{(\delta)}\lesssim \frac{\log (n+1)}{(n+1)^{\delta}},\notag
\end{equation}
and so (\ref{dndelta}) holds.
\end{proof} 

In \cite{xu}, Y. Xu proved that if $f\in C(\overline {\Omega}) $ the $(C,1)$ means $S^{(1)}_{n}(f)$ converges uniformly to $f$ on $\overline {\Omega}$. The degree of approximation of $(C,1)$ means of hexagonal Fourier series of functions belong to  $C(\overline {\Omega}) $ was studied in \cite{guven jca}, \cite{guven mia} and \cite{guven cmft}.

The following theorem gives the main estimate for the degree of approximation of $(C,\delta)$ means of hexagonal Fourier series.
\begin{theorem}
    For $f \in C(\overline{\Omega})$ the estimate
    \begin{equation}\label{main}
        \left \|f-S^{(\delta)}_{n}(f)\right \|_{C(\overline {\Omega})}\lesssim \begin{cases}
        \log (n+1) \omega_{f} \left (\frac {\log (n+1)}{(n+1)^{\delta}}\right ) , &0<\delta<1\\
        \ \omega_{f} \left (\frac {(\log (n+1))^{2}}{n+1}\right ),&\delta \geq 1
        \end{cases}
    \end{equation}
    holds.
\end{theorem}
\begin{proof}
    Let $f \in C(\overline{\Omega}).$ By (\ref{cesarodelta}) and (\ref{kernel}) we have
    \begin{equation}
        f(\textbf{t})-S_{n}^{(\delta)}(f)(\textbf{t})=\frac{1}{\left |\Omega \right |}\int_{\Omega}\left (f(\textbf{t})-f(\textbf{t}-\textbf{s})\right )K_{n}^{(\delta)}(\textbf{s})d\textbf{s}.\notag
    \end{equation}
    If we set
\begin{equation}
    d_{n}^{(\delta)}:= \frac{1}{\left |\Omega \right |}\int_{\Omega}\left \|\textbf{t}\right \|\left |K^{(\delta)}_{n}(\textbf{t})\right |d\textbf{t}\notag,
\end{equation}
we obtain
\begin{align}
    \left | f(\textbf{t})-S_{n}^{(\delta)}(f)(\textbf{t})\right |&\leq \frac{1}{\left |\Omega \right |}\int_{\Omega}\left |f(\textbf{t})-f(\textbf{t}-\textbf{s})\right |\left |K_{n}^{(\delta)}(\textbf{s})\right |d\textbf{s}\notag \\
    &\leq \frac{1}{\left |\Omega \right |}\int_{\Omega}\omega_{f} \left (\left \|\textbf{s}\right \|\right )\left |K_{n}^{(\delta)}(\textbf{s})\right |d\textbf{s}\notag \\
    &= \frac{1}{\left |\Omega \right |}\int_{\Omega}\omega_{f} \left (\frac{\left \|\textbf{s}\right \|}{ d_{n}^{(\delta)}} d_{n}^{(\delta)}\right )\left |K_{n}^{(\delta)}(\textbf{s})\right |d\textbf{s}\notag \\
    &\leq \omega_{f} \left ( d_{n}^{(\delta)}\right ) \frac{1}{\left |\Omega \right |}\int_{\Omega}\left (1+\frac{\left \|\textbf{s}\right \|}{ d_{n}^{(\delta)}}\right )\left |K_{n}^{(\delta)}(\textbf{s})\right |d\textbf{s}\notag \\
    &=\omega_{f} \left ( d_{n}^{(\delta)}\right ) \left (\frac{1}{\left |\Omega \right |}\int_{\Omega}\left |K_{n}^{(\delta)}(\textbf{s})\right |d\textbf{s}+\frac{1}{d_{n}^{\delta}}\frac{1}{\left |\Omega \right |}\int_{\Omega}\left \|\textbf{s}\right \|\left |K^{(\delta)}_{n}(\textbf{s})\right |d\textbf{s}\right )\notag \\
    &=\omega_{f} \left (d_{n}^{(\delta)}\right )\left (1+\frac{1}{\left |\Omega \right |}\int_{\Omega}\left |K_{n}^{(\delta)}(\textbf{s})\right |d\textbf{s}\right ).\notag
\end{align}
Hence, for $0<\delta <1$, (\ref{main}) follows from (\ref{Kndelta}) and (\ref{dndelta}). It follows from \cite{xu} that
\begin{equation}\label{kn1}
   \int_{\Omega}\left |K^{(1)}_{n}(\textbf{t})\right |d\textbf{t}\lesssim 1,
\end{equation}
and from \cite{guven mia}
\begin{equation}
    \int_{\Omega}\left \|\textbf{t}\right \|\left |K^{(1)}_{n}(\textbf{t})\right |d\textbf{t}\lesssim \frac{\left (\log (n+1)\right )^{2}}{n+1}.\notag
\end{equation}
Thus we get (\ref{main}) for $\delta=1$.
For $\delta >1$, we use the equality (see \cite{ulyanov})
\begin{equation}
    S_{n}^{(\delta)}(f)(\textbf{t})=\frac{1}{A_{n}^{\delta}}\sum_{k=0}^{n}A_{n-k}^{\delta-1}S_{k}(f)(\textbf{t})=\frac{1}{A_{n}^{\delta}}\sum_{k=0}^{n}A_{n-k}^{\delta-2}A_{k}^{1}S_{k}^{(1)}(f)(\textbf{t}).\notag
\end{equation}
Thus,
\begin{equation}
    f(\textbf{t})- S_{n}^{(\delta)}(f)(\textbf{t})=\frac{1}{A_{n}^{\delta}}\sum_{k=0}^{n}A_{n-k}^{\delta-2}A_{k}^{1}\left (f(\textbf{t})-S_{k}^{(1)}(f)(\textbf{t})\right ).\notag
\end{equation}
Hence, considering (\ref{main}),
\begin{align}
    \left | f(\textbf{t})- S_{n}^{(\delta)}(f)(\textbf{t})\right |&\leq \frac{1}{A_{n}^{\delta}}\sum_{k=0}^{n}A_{n-k}^{\delta-2}A_{k}^{1}\left |f(\textbf{t})-S_{k}^{(1)}(f)(\textbf{t})\right |\notag \\
    &\lesssim \frac{1}{A_{n}^{\delta}}\sum_{k=0}^{n}A_{n-k}^{\delta-2}A_{k}^{1}\omega_{f} \left (\frac{(\log (k+1))^{2}}{k+1}\right )\notag \\
    &\lesssim \frac{1}{(n+1)^{\delta}}\sum_{k=0}^{n}(n-k+1)^{\delta-2}(k+1)\omega_{f} \left (\frac{(\log (k+1))^{2}}{k+1}\right )\notag \\
    &=\frac{1}{(n+1)^{\delta}}\sum_{k=0}^{n}(n-k+1)^{\delta-2}(k+1)\notag \\
    &\times \omega_{f} \left (\frac{(\log (n+1))^{2}}{n+1}\frac{n+1}{(\log (n+1))^{2}}\frac{(\log (k+1))^{2}}{k+1}\right )\notag \\
    &\leq \frac{1}{(n+1)^{\delta}}\omega_{f} \left (\frac{(\log (n+1))^{2}}{n+1}\right )\sum_{k=0}^{n}(n-k+1)^{\delta-2}(k+1)\notag \\ 
    &\times \left (1+\frac{n+1}{(\log (n+1))^{2}}\frac{(\log (k+1))^{2}}{k+1}\right )\notag \\
    &\leq \frac{1}{(n+1)^{\delta}}\omega_{f} \left (\frac{(\log (n+1))^{2}}{n+1}\right )\sum_{k=0}^{n}(n-k+1)^{\delta-2}(k+1)\left (1+\frac{n+1}{k+1}\right )\notag \\
    &\lesssim \frac{1}{(n+1)^{\delta-1}}\omega_{f} \left (\frac{(\log (n+1))^{2}}{n+1}\right )\sum_{k=0}^{n}(n-k+1)^{\delta-2}\notag \\
    &\leq \omega_{f} \left (\frac{(\log (n+1))^{2}}{n+1}\right ).\notag
\end{align}
\end{proof}
The analogue of (\ref{main}) for classical Fourier series was obtained in \cite{taberski}.

\section{Approximation by Abel-Poisson Means on Hexagonal Domain}
Abel-Poisson means of series (\ref{series}) are defined by
\begin{equation}
    U_{r}(f)(\textbf{t}):=\sum_{k=0}^{\infty}\sum_{\textbf{j}\in \mathbb{J}_{k}}r^{k}\widehat{f}_{\textbf{j}}\phi_{\textbf{j}}(\textbf{t})=\frac{1}{\left |\Omega\right |}\int_{\Omega}f(\textbf{t}-\textbf{s})P_{r}(\textbf{s})d\textbf{s},\hspace{3mm} 0\leq r< 1,\notag
\end{equation}
where
\begin{equation}
    \mathbb{J}_{k}:=\mathbb{H}_{k}\setminus \mathbb{H}_{k-1}\notag
\end{equation}
and 
\begin{equation}
    P_{r}(\textbf{t}):=\sum_{k=0}^{\infty}\sum_{\textbf{j}\in \mathbb{J}_{k}}r^{k}\phi_{\textbf{j}}(\textbf{t})\notag
\end{equation}
is the Poisson kernel. The Poisson kernel is non-negative, satisfies
\begin{equation}\label{poisson}
    \frac{1}{\left |\Omega\right |}\int_{\Omega}P_{r}(\textbf{t})d\textbf{t}=1,
\end{equation}
and has the compact formula
\begin{align}
   P_{r}(\textbf{t})&=\frac{(1-r)^{3}(1-r^{3})}{q_{r}\left (\frac{2\pi (t_{1}-t_{2})}{3}\right ) q_{r}\left (\frac{2\pi (t_{2}-t_{3})}{3}\right )q_{r}\left (\frac{2\pi (t_{3}-t_{1})}{3}\right )} \notag \\
   &+\frac{r(1-r)^{2}}{q_{r}\left (\frac{2\pi (t_{1}-t_{2})}{3}\right )q_{r}\left (\frac{2\pi (t_{2}-t_{3})}{3}\right )}\notag \\
   &+\frac{r(1-r)^{2}}{q_{r}\left (\frac{2\pi (t_{2}-t_{3})}{3}\right )q_{r}\left (\frac{2\pi (t_{3}-t_{1})}{3}\right )}\notag \\
   &+\frac{r(1-r)^{2}}{q_{r}\left (\frac{2\pi (t_{3}-t_{1})}{3}\right )q_{r}\left (\frac{2\pi (t_{1}-t_{2})}{3}\right )}\notag
\end{align}
for $\textbf{t}=(t_{1},t_{2},t_{3})\in \mathbb{R}_{H}^{3},$ where $q_{r}(x)=1-2r\cos x+r^{2}$ (\cite{xu}). Also, it is easy to see that
\begin{align}\label{poisson1}
     P_{r}(\textbf{t})&\leq \frac{2(1-r)^{2}}{q_{r}\left (\frac{2\pi (t_{1}-t_{2})}{3}\right )q_{r}\left (\frac{2\pi (t_{2}-t_{3})}{3}\right )}+\frac{2(1-r)^{2}}{q_{r}\left (\frac{2\pi (t_{2}-t_{3})}{3}\right )q_{r}\left (\frac{2\pi (t_{3}-t_{1})}{3}\right )}\notag \\
     &+\frac{2(1-r)^{2}}{q_{r}\left (\frac{2\pi (t_{3}-t_{1})}{3}\right )q_{r}\left (\frac{2\pi (t_{1}-t_{2})}{3}\right )}
\end{align}
and
\begin{equation}\label{poisson2}
    \frac{(1-r)^{2}}{q_{r}(x)q_{r}(y)}=\frac{1}{(1+r)^{2}}p_{r}(x)p_{r}(y),
\end{equation}
where
\begin{equation}
    p_{r}(x)=\frac{1-r^{2}}{q_{r}(x)}\notag
\end{equation}
is the classical Poisson kernel. The classical Poisson kernel satisfies the inequalities (\cite [pp. 96-97]{zygmund})
\begin{equation}\label{poisson3}
    p_{r}(x)\leq \frac{2}{1-r},\hspace{3mm}0\leq x\leq \pi
\end{equation}
and
\begin{equation}\label{poisson4}
    p_{r}(x)\lesssim \frac{1-r}{x^{2}},\hspace{3mm}0< x\leq \pi.
\end{equation}

We have the following estimate for Abel-Poisson means of hexagonal Fourier series of $H-$periodic continuous functions.

\begin{theorem}
    For $f \in C(\overline{\Omega})$ the estimate
    \begin{equation}\label{main2}
        \left \|f-U_{r}(f)\right \|_{C(\overline {\Omega})}\lesssim \omega_{f}\left ((1-r)\left |\log (1-r)\right | \right )
        \end{equation}
        holds as $r\rightarrow 1^{-}.$
\end{theorem}
\begin{proof}
 If we set
   \begin{equation}
       \lambda_{n}^{(r)}:=\frac{1}{\left |\Omega\right |}\int_{\Omega}\left \|\textbf{t}\right \|P_{r}(\textbf{t})d\textbf{t},\notag
   \end{equation}
   by (\ref{poisson}) 
   \begin{align}
       \left |f(\textbf{t})-U_{r}(f)(\textbf{t})\right |&\leq \frac{1}{\left |\Omega\right |}\int_{\Omega}\left |f(\textbf{t})-f(\textbf{t}-\textbf{s})\right |P_{r}(\textbf{s})d\textbf{s}\notag\\
       &\leq \frac{1}{\left |\Omega\right |}\int_{\Omega}\omega_{f}\left (\left \|\textbf{s}\right \|\right )P_{r}(\textbf{s})d\textbf{s}\notag \\
       &=\frac{1}{\left |\Omega\right |}\int_{\Omega}\omega_{f}\left (\frac{\left \|\textbf{s}\right \|}{ \lambda_{n}^{(r)}} \lambda_{n}^{(r)}\right )P_{r}(\textbf{s})d\textbf{s}\notag\\
       &\leq \omega_{f}\left (\lambda_{n}^{(r)}\right )\frac{1}{\left |\Omega\right |}\int_{\Omega}\left (1+\frac{\left \|\textbf{s}\right \|}{\lambda_{n}^{(r)}}\right )P_{r}(\textbf{s})d\textbf{s}\notag \\
       &=\omega_{f}\left (\lambda_{n}^{(r)}\right )\left (\frac{1}{\left |\Omega\right |}\int_{\Omega}P_{r}(\textbf{s})d\textbf{s}+\frac{1}{\lambda_{n}^{(r)}}\frac{1}{\left |\Omega\right |}\int_{\Omega}\left \|\textbf{s}\right \|P_{r}(\textbf{s})d\textbf{s}\right )\notag \\
       &=2\omega_{f}\left (\lambda_{n}^{(r)}\right ).\notag
   \end{align}
  Since the function $\omega_{f}\left (\cdot \right )$ is non-decreasing we have to estimate the quantity $\lambda_{n}^{(r)}.$ If we set 
  \begin{align}
     Q_{r}(\textbf{t})&:= \frac{2(1-r)^{2}}{q_{r}\left (\frac{2\pi (t_{1}-t_{2})}{3}\right )q_{r}\left (\frac{2\pi (t_{2}-t_{3})}{3}\right )}+\frac{2(1-r)^{2}}{q_{r}\left (\frac{2\pi (t_{2}-t_{3})}{3}\right )q_{r}\left (\frac{2\pi (t_{3}-t_{1})}{3}\right )}\notag \\
     &+\frac{2(1-r)^{2}}{q_{r}\left (\frac{2\pi (t_{3}-t_{1})}{3}\right )q_{r}\left (\frac{2\pi (t_{1}-t_{2})}{3}\right )},\notag
\end{align}
(\ref{poisson1}) implies that it will be sufficient to estimate the integral
\begin{equation}
    \int_{\Delta}\left \|\textbf{t}\right \|Q_{r}(\textbf{t})d\textbf{t}.\notag
\end{equation}
If we use the change of variables (\ref{transform1}) and (\ref{transform2}) estimating the integral 
\begin{equation}
    \int_{\Gamma}u_{1}Q_{r}^{*}(u_{1},u_{2})du_{1}du_{2},\notag
\end{equation}
where
\begin{align}
    Q_{r}^{*}(u_{1},u_{2})&:=p_{r}\left (\pi (u_{1}+u_{2})\right )p_{r}\left (2\pi u_{2}\right )\notag \\
    &+p_{r}\left (\pi (u_{1}-u_{2})\right )p_{r}\left (\pi (u_{1}+u_{2})\right )\notag \\
    &+p_{r}\left (\pi (u_{1}-u_{2})\right )p_{r}\left (2\pi u_{2}\right )\notag
\end{align}
will be sufficient. We write the triangle $\Gamma$ as $\Gamma =\Gamma_{1}^{*}\cup \Gamma_{2}^{*}\cup \Gamma_{3}^{*}$, where
\begin{align}
    \Gamma_{1}^{*}&:=\lbrace (u_{1},u_{2})\in \Gamma :u_{1}\leq 1-r\rbrace \notag \\
    \Gamma_{2}^{*}&:=\lbrace (u_{1},u_{2})\in \Gamma :u_{1}\geq 1-r, u_{2}\leq \frac{1-r}{3}\rbrace \notag \\
    \Gamma_{3}^{*}&:=\lbrace (u_{1},u_{2})\in \Gamma :u_{1}\geq 1-r, u_{2}\geq \frac{1-r}{3}\rbrace, \notag
\end{align}
and hence
\begin{equation}
     \int_{\Gamma}u_{1}Q_{r}^{*}(u_{1},u_{2})du_{1}du_{2}=\sum_{k=1}^{3}\left (\int_{\Gamma_{k}^{*}}u_{1}Q_{r}^{*}(u_{1},u_{2})du_{1}du_{2}\right ).\notag
\end{equation}
By (\ref{poisson3}),
\begin{align}
    \int_{\Gamma_{1}^{*}}u_{1}Q_{r}^{*}(u_{1},u_{2})du_{1}du_{2}&=\int_{0}^{\frac{1-r}{3}}\left (\int_{3u_{2}}^{1-r}u_{1}Q_{r}^{*}(u_{1},u_{2})du_{1}\right )du_{2}\notag \\
    &\lesssim \frac{1}{(1-r)^{2}}\int_{0}^{\frac{1-r}{3}}\left (\int_{3u_{2}}^{1-r}u_{1}du_{1}\right )du_{2}\notag \\
    &\lesssim 1-r.\notag
\end{align}
If we consider (\ref{poisson3}) and (\ref{poisson4}) and taking into account the inequality $u_{1}-u_{2}\geq \frac{2}{3}u_{1}$,
\begin{align}
     \int_{\Gamma_{2}^{*}}u_{1}Q_{r}^{*}(u_{1},u_{2})du_{1}du_{2}&=\int_{0}^{\frac{1-r}{3}}\left (\int_{1-r}^{1}u_{1}Q_{r}^{*}(u_{1},u_{2})du_{1}\right )du_{2}\notag \\
     &\lesssim \int_{0}^{\frac{1-r}{3}}\left (\int_{1-r}^{1}u_{1} \left (\frac{1}{\pi ^{2}(u_{1}+u_{2})^{2}} +\frac{1}{\pi ^{2}(u_{1}-u_{2})^{2}}\right )du_{1}\right )du_{2}\notag \\
     &\leq  \int_{0}^{\frac{1-r}{3}}\left (\int_{1-r}^{1}\frac{1}{u_{1}}du_{1} \right )du_{2}=\frac{1-r}{3} \left |\log (1-r)\right |\notag \\
     &\leq (1-r)\left |\log (1-r)\right |.\notag
\end{align}
Now considering (\ref{poisson4}) again and using $u_{1}-u_{2}\geq \frac{2}{3}u_{1}$, we get 
\begin{align}
     \int_{\Gamma_{3}^{*}}u_{1}Q_{r}^{*}(u_{1},u_{2})du_{1}du_{2}&=\int_{\frac{1-r}{3}}^{\frac{1}{3}}\left (\int_{3u_{2}}^{1}u_{1}Q_{r}^{*}(u_{1},u_{2})du_{1}\right )du_{2}\notag \\
     &\lesssim (1-r)^{2} \int_{\frac{1-r}{3}}^{\frac{1}{3}}\left (\int_{3u_{2}}^{1}\frac{1}{u_{1}u_{2}^{2}}du_{1}\right )du_{2}\notag \\
     &\lesssim (1-r)\left |\log (1-r)\right |.\notag
\end{align}
\end{proof}
The analogue of Theorem 4 was proved in \cite{natanson} for Abel-Poisson means of classical Fourier series.

\end{document}